\documentclass[submission,copyright,creativecommons]{eptcs}
\providecommand{\event}{QPL 2017}
\usepackage{amsmath,amsthm,amssymb,stmaryrd,enumitem,color}

\theoremstyle{definition}
\newtheorem{theorem}{Theorem}
\newtheorem{corollary}[theorem]{Corollary}
\newtheorem{lemma}[theorem]{Lemma}
\newtheorem{proposition}[theorem]{Proposition}
\theoremstyle{definition}
\newtheorem{definition}[theorem]{Definition}
\newtheorem{example}[theorem]{Example}
\newtheorem{remark}[theorem]{Remark}

\DeclareMathOperator{\supp}{supp}
\newcommand{\cat}[1]{\ensuremath{\mathbf{#1}}}

\newcommand{\id}[1][]{\ensuremath{\mathrm{id}_{#1}}}
\newcommand{\chronprec}{\ll}
\newcommand{\causprec}{\prec}
\newcommand{\closure}[1]{\overline{{#1}}} 
\newcommand{\restrict}[1]{\ensuremath{|_{#1}}}
\newcommand{\inprod}[2]{\langle #1 \mid #2 \rangle}

\usepackage{tikz}
\usetikzlibrary{arrows,arrows.meta,calc,matrix,shapes,cd,patterns,decorations.pathreplacing}

\newenvironment{pic}[1][]
{\begin{aligned}\begin{tikzpicture}[#1]}
{\end{tikzpicture}\end{aligned}}

\def\strarr{length=2pt, width=3pt}
\tikzset{arrow/.style={decoration={
    markings,
    mark=at position #1 with \arrow{>[\strarr]}},
    postaction=decorate},
    reverse arrow/.style={decoration={
    markings,
    mark=at position #1 with {{\arrow{<[\strarr]}}}},
    postaction=decorate}
}

\tikzstyle{dot}=[circle, draw=black, fill=black!25, inner sep=.4ex]
\tikzstyle{blackdot}=[dot, fill=black!50]
\tikzstyle{whitedot}=[dot, fill=white]

\newif\ifvflip\pgfkeys{/tikz/vflip/.is if=vflip}
\newif\ifhflip\pgfkeys{/tikz/hflip/.is if=hflip}
\newif\ifhvflip\pgfkeys{/tikz/hvflip/.is if=hvflip}
\newlength\morphismheight
\newlength\wedgewidth
\tikzset{width/.initial=1mm}
\makeatletter
\tikzstyle{morphism}=[font=\small,morphismshape]
\pgfdeclareshape{morphismshape}
{
    \savedanchor\centerpoint
    {
        \pgf@x=0pt
        \pgf@y=0pt
    }
    \anchor{center}{\centerpoint}
    \anchorborder{\centerpoint}
    \saveddimen\overallwidth
    {
        \pgfkeysgetvalue{/tikz/width}{\minwidth}
        \pgf@x=\wd\pgfnodeparttextbox
        \ifdim\pgf@x<\minwidth
            \pgf@x=\minwidth
        \fi
    }
    \savedanchor{\upperrightcorner}
    {
        \pgf@y=.5\ht\pgfnodeparttextbox
        \advance\pgf@y by -.5\dp\pgfnodeparttextbox
        \pgf@x=.5\wd\pgfnodeparttextbox
    }
    \anchor{north}
    {
        \pgf@x=0pt
        \pgf@y=0.5\morphismheight
    }
    \anchor{north east}
    {
        \pgf@x=\overallwidth
        \multiply \pgf@x by 2
        \divide \pgf@x by 5
        \pgf@y=0.5\morphismheight
    }
    \anchor{east}
    {
        \pgf@x=\overallwidth
        \divide \pgf@x by 2
        \advance \pgf@x by 5pt
        \pgf@y=0pt
    }
    \anchor{west}
    {
        \pgf@x=-\overallwidth
        \divide \pgf@x by 2
        \advance \pgf@x by -5pt
        \pgf@y=0pt
    }
    \anchor{north west}
    {
        \pgf@x=-\overallwidth
        \multiply \pgf@x by 2
        \divide \pgf@x by 5
        \pgf@y=0.5\morphismheight
    }
    \anchor{south east}
    {
        \pgf@x=\overallwidth
        \multiply \pgf@x by 2
        \divide \pgf@x by 5
        \pgf@y=-0.5\morphismheight
    }
    \anchor{south west}
    {
        \pgf@x=-\overallwidth
        \multiply \pgf@x by 2
        \divide \pgf@x by 5
        \pgf@y=-0.5\morphismheight
    }
    \anchor{south}
    {
        \pgf@x=0pt
        \pgf@y=-0.5\morphismheight
    }
    \anchor{text}
    {
        \upperrightcorner
        \pgf@x=-\pgf@x
        \pgf@y=-\pgf@y
    }
    \backgroundpath
    {
    \begin{scope}
        \pgfkeysgetvalue{/tikz/fill}{\morphismfill}
        \pgfsetstrokecolor{black}
        \pgfsetlinewidth{.7pt}
        \begin{scope}
        \pgfsetstrokecolor{black}
        \pgfsetfillcolor{white}
        \pgfsetlinewidth{.7pt}
                \ifhflip
                    \pgftransformyscale{-1}
                \fi
                \ifvflip
                    \pgftransformxscale{-1}
                \fi
                \ifhvflip
                    \pgftransformxscale{-1}
                    \pgftransformyscale{-1}
                \fi
                \pgfpathmoveto{\pgfpoint
                    {-0.5*\overallwidth-5pt}
                    {0.5*\morphismheight}}
                \pgfpathlineto{\pgfpoint
                    {0.5*\overallwidth+5pt}
                    {0.5*\morphismheight}}
                \pgfpathlineto{\pgfpoint
                    {0.5*\overallwidth + \wedgewidth}
                    {-0.5*\morphismheight}}
                \pgfpathlineto{\pgfpoint
                    {-0.5*\overallwidth-5pt}
                    {-0.5*\morphismheight}}
                \pgfpathclose
                \pgfusepath{fill,stroke}
        \end{scope}
    \end{scope}
    }
}
\makeatother

\title{Space in Monoidal Categories}

\author{Pau Enrique Moliner
    \institute{University of Edinburgh}
    \email{pau.enrique.moliner@ed.ac.uk}
  \and Chris Heunen 
    \institute{University of Edinburgh}\thanks{Supported by EPSRC Fellowship EP/L002388/1. We thank Prakash Panangaden and Manny Reyes for discussions.}
    \email{chris.heunen@ed.ac.uk}
  \and Sean Tull
    \institute{University of Oxford}\thanks{Supported by EPSRC Studentship OUCL/2014/SET.}
    \email{sean.tull@cs.ox.ac.uk}
}
\def\titlerunning{Space in Monoidal Categories}
\def\authorrunning{P. Enrique Moliner, C. Heunen \& S. Tull}
\begin{document}
\maketitle

\begin{abstract}
  The category of Hilbert modules may be interpreted as a naive quantum field theory over a base space. Open subsets of the base space are recovered as idempotent subunits, which form a meet-semilattice in any firm braided monoidal category. There is an operation of restriction to an idempotent subunit: it is a graded monad on the category, and has the universal property of algebraic localisation. Spacetime structure on the base space induces a closure operator on the idempotent subunits. Restriction is then interpreted as spacetime propagation. This lets us study relativistic quantum information theory using methods entirely internal to monoidal categories. As a proof of concept, we show that quantum teleportation is only successfully supported on the intersection of Alice and Bob's causal future.
\end{abstract}

\section{Introduction}

Categorical quantum mechanics reveals conceptual foundations of quantum theory by abstracting from Hilbert spaces to monoidal categories with various operationally motivated properties~\cite{abramskycoecke:categoricalsemantics,coeckekissinger:cqm,heunenvicary:cqm}. This minimal structure is already enough to recover many features of quantum theory. For example, endomorphisms of the tensor unit, also called scalars, play the role that complex numbers do for Hilbert spaces. 
This article highlights the spatial structure of morphisms into the tensor unit.

To illustrate this, consider Hilbert modules, which roughly replace the complex numbers in the definition of Hilbert space with complex-valued functions over a base space $X$. They form a monoidal category, whose morphisms into the tensor unit encode the spatial structure of $X$.
That was a very algebraic definition of Hilbert modules, but there is also a geometric interpretation. Roughly, Hilbert modules are equivalent to fields of Hilbert spaces over $X$: a bundle of Hilbert spaces $H_t$ that vary continuously with $t \in X$. Thus we may think of this category as (a naive version of) quantum field theory. We can study such settings entirely within monoidal categories.

We first concentrate on recovering the base space $X$ using purely categorical structure. The answer is given by subobjects $S$ of the tensor unit that are idempotent, in the sense that $S \otimes S \simeq S$. Such idempotent subunits correspond to open subsets of $X$. 
Idempotent subunits form a meet-semilattice in any braided monoidal category satisfying a mild condition we call `firmness'.
This gives a way to talk about `where' morphisms are `supported', and to restrict morphisms to smaller regions. We will prove that the operation of restriction is a graded monad, and that it has the universal property of algebraic localisation.

A natural next question is: what extra structure on $X$ could translate into categorical terms? So far we have thought of $X$ as a mere space, but what if it came equipped with a spacetime structure? 
We will show that this gives rise to a closure operator on the idempotent subunits, letting us model causal relationships categorically: the restriction operator is interpreted as spacetime propagation. As a proof of concept, we show that quantum teleportation is only succesfully supported on the intersection of Alice and Bob's causal futures.

\section{Hilbert modules}\label{sec:hilbertmodules}

We start with a brief section detailing our leading example, the category of Hilbert modules. For more information we refer to Appendix~\ref{sec:tensorsofhilbertmodules} and~\cite{lance1995hilbert,heunenreyes:frobenius}.

\begin{definition}\label{def:hilbertmodule}
  Fix a locally compact Hausdorff space $X$. It induces a commutative C*-algebra 
  \[
    C_0(X)=\{f \colon X \to \mathbb{C} \text{ continuous} \mid \forall \varepsilon>0\; \exists K \subseteq X \text{ compact} \colon |f(X \setminus K)|<\varepsilon \}\text.
  \]
  A \emph{Hilbert module} is a $C_0(X)$-module $E$ with a map $\inprod{-}{-} \colon E \times E \to C_0(X)$ that is $C_0(X)$-linear in the second variable, satisfies $\inprod{x}{y}=\inprod{y}{x}^*$, $\inprod{x}{x}\geq 0$ with equality only if $x=0$, and makes $E$ complete in the norm $\|x\|_E = \sup_{t \in X} \inprod{x}{x}(t)$~\cite{lance1995hilbert}.

  A function $f \colon E \to F$ between Hilbert modules is \emph{bounded} if $\|f(x)\|_F \leq \|f\|\|x\|_E$ for some $\|f\| \in \mathbb{R}$.
  Hilbert modules  and bounded $C_0(X)$-linear maps form a symmetric monoidal category $\cat{Hilb}_{C_0(X)}$: $E \otimes F$ is the completion of the algebraic tensor product over $C_0(X)$ under $\inprod{x \otimes y}{x' \otimes y'} = \inprod{x}{x'} \inprod{y}{y'}$~\cite[Proposition~2.2]{heunenreyes:frobenius}. The tensor unit is the Hilbert module $C_0(X)$ with $\inprod{f}{g}(t) = f(t)^* g(t)$.
\end{definition}

\begin{example}\label{ex:constanthilbertmodule}
  If $X=1$ then a Hilbert module is simply a Hilbert space.
  More generally, for any $t \in X$ there is a monoidal functor $\cat{Hilb}_{C_0(X)} \to \cat{Hilb}$~\cite[Proposition~2.5]{heunenreyes:frobenius}.
  In the other direction, there is a monoidal functor $C_0(X,-) \colon \cat{Hilb} \to \cat{Hilb}_{C_0(X)}$ that sends a Hilbert space $H$ to the Hilbert module
  \[
    C_0(X,H) = \{ f \colon X \to H \text{ continuous} \mid \forall \varepsilon>0 \; \exists K \subseteq X \text{ compact} \colon \|f(X \setminus K)\|<\varepsilon\}\text.
  \]
\end{example}

In particular, if $H$ is finite-dimensional and so has a \emph{dual object}~\cite{heunenvicary:cqm} $H^*$ in $\cat{Hilb}$, then $C_0(X,H)$ and $C_0(X,H^*)$ are dual objects in $\cat{Hilb}_{C_0(X)}$.
More generally, the dual objects in $\cat{Hilb}_{C_0(X)}$ are the finitely presented projective Hilbert modules~\cite[Theorem~5.5]{heunenreyes:frobenius}.

The above discussion is very algebraic in nature, and therefore lends itself well to categorical treatment. However, there is also a geometric description of Hilbert modules.

\begin{definition}
  A \emph{field of Hilbert spaces} is a continuous map $p \colon E \to X$ of topological spaces such that:
  \begin{itemize}
  \item all fibres $p^{-1}(t)$ for $t \in X$ are Hilbert spaces;
  \item addition is a continuous function $\{(x,y) \in E^2 \mid p(x)=p(y)\} \to E$;
  \item scalar multiplication is a continuous function $\mathbb{C} \times E \to E$;
  \item the inner product is a continuous function $\{(x,y) \in E^2 \mid p(x)=p(y)\} \to \mathbb{C}$;
  \item each $x_0 \in E$ has a continuous local section $s \colon U_0 \to E$ with $s(p(x_0))=x_0$ and $U_0 \subseteq X$, 
    and $x_0$ has a neighbourhood basis $p^{-1}(U) \cap \{x \in E \mid \forall t \in U \colon \|x - s(p(x))\|_t < \varepsilon \}$ of neighbourhoods $U\subseteq X$ of $p(x_0)$ and $\varepsilon>0$.
  \end{itemize}  
\end{definition}

  
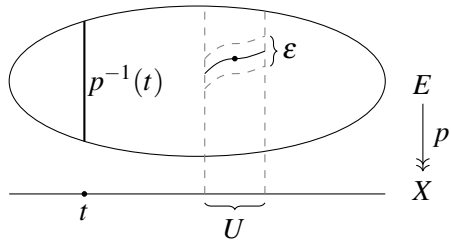
\[\begin{tikzpicture}
    \draw (0,0) to (5,0);
    \draw (2.5,1.5) ellipse (2.5 and 1);
    \draw[->>] (5.5,1.2) node[above]{$E$} to node[right]{$p$} (5.5,.3) node[below]{$X$};
    \draw[fill] (1,0) circle (.03) node[below]{$t$};
    \draw[thick] (1,.7) to node[right=-1mm]{$p^{-1}(t)$} (1,2.3);
      \draw[fill] (3,1.8) circle (.03);
      \draw[decoration={brace,mirror,raise=.5},decorate] (2.6,-.1) to node[below=1mm]{$U$} (3.4,-.1);    
      \draw[dashed,gray] (2.6,0) to (2.6,2.5);
      \draw[dashed,gray] (3.4,0) to (3.4,2.5);
      \draw (2.6, 1.6) to[in=180] (3,1.8) to[out=0,in=-160] (3.4,1.9);
      \draw[dashed,gray] (2.6, 1.8) to[in=180] (3,2.0) to[out=0,in=-160] (3.4,2.1);
      \draw[dashed,gray] (2.6, 1.4) to[in=180] (3,1.6) to[out=0,in=-160] (3.4,1.7);
      \draw[decoration={brace,mirror,raise=2},decorate] (3.4,1.7) to node[right=1mm]{$\varepsilon$} (3.4,2.1);
  \end{tikzpicture}\]
Intuitively, there is a Hilbert space for each $t \in X$, that `varies continuously' with $t$. 
By the following we may think of Hilbert modules as naive quantum field theories, see~\cite[Theorem~4.7]{heunenreyes:frobenius} for a proof.

\begin{theorem}
  The category $\cat{Hilb}_{C_0(X)}$ is equivalent to that of fields of Hilbert spaces over $X$ (with morphisms being fibre-wise linear bundle maps).
\end{theorem}

\section{Subunits}

Can we recognise the spatial structure of $X$ from the categorical structure of $\cat{Hilb}_{C_0(X)}$? The answer is yes, and this section discusses how.
Recall that a \emph{subobject} of an object $E$ is an equivalence class of monomorphisms $s \colon S \rightarrowtail E$, where $s$ and $t$ are identified if they factor through each other. Call subobjects of the tensor unit $I$ \emph{subunits}.

\begin{definition}\label{def:idempotent}
  A subobject $s \colon S \rightarrowtail E$ is \emph{idempotent} when $s \otimes \id[S] \colon S \otimes S \to E \otimes S$ is an isomorphism. 
  An object $E$ is \emph{firm} for a subunit $s$ when $s \otimes \id[E]$ is monic.
  We call a monoidal category \emph{firm} when $S$ is firm for $t$ for any subunits $s \colon S \rightarrowtail I$ and $t \colon T \rightarrowtail I$.
\end{definition}

Any compact category is firm. 
Cartesian categories are firm, any subunit is idempotent and corresponds precisely to a subterminal object; hence in the category of sheaves over a topological space $X$, the subunits are precisely the open subsets of $X$.
In Appendix~\ref{sec:firmrings} we discuss algebraic examples, but our motivating example is the following; for a proof see Appendix~\ref{sec:tensorsofhilbertmodules}.

\begin{theorem}\label{thm:hilbertmodules}
  The monoidal category $\cat{Hilb}_{C_0(X)}$ is firm, and the idempotent subunits are
  \begin{equation}\label{eq:subunithilbertmodules}
    \{ f \in C_0(X) \mid f(X \setminus U)=0 \} \simeq C_0(U)
  \end{equation}
  for open subsets $U \subseteq X$.
\end{theorem}

Many spatial properties of topological spaces like $X$ hold in more general monoidal categories than $\cat{Hilb}_{C_0(X)}$.
This section shows that in an arbitrary firm braided monoidal category, the idempotent subunits form a meet-semilattice. Many results are adapted from the work of Boyarchenko and Drinfeld~\cite{boyarchenkodrinfeld:idempotents}, including the following useful observation.

\begin{lemma} 
  \label{lem:splitmonosubunits}
  Let $m \colon E \to F$ and $e \colon F \to E$ satisfy $e \circ m = \id[E]$, and $s \colon S \rightarrowtail I$ be an idempotent subunit.
  If $\id[F] \otimes s$ is an isomorphism, then so is $\id[E] \otimes s$.
\end{lemma}
\begin{proof}
  Both rows below compose to the identity, and the middle vertical arrow is an isomorphism.
  \[\begin{tikzcd}
    E \otimes S \rar{m \otimes \id[S]} \dar{\id[E] \otimes s}
    & F \otimes S \rar{e \otimes \id[S]} \arrow{d}{\id[F] \otimes s}[swap]{\simeq}
    & E \otimes S \dar{\id[E] \otimes s} \\
    E \otimes I \rar{m \otimes \id[I]}
    & F \otimes I \rar{e \otimes \id[I]}
    & E \otimes I
  \end{tikzcd}\]
  Hence $\id[E] \otimes s$ is an isomorphism with inverse $(e \otimes \id[S]) \circ (\id[F] \otimes s)^{-1} \circ (m \otimes \id[I])$.
\end{proof}

\begin{proposition} \label{prop:subunits_less}
  One idempotent subunit $s$ factors through another $t$ if and only if $\id[S] \otimes t$ is an isomorphism, or equivalently, $t \otimes \id[S]$ is an isomorphism. 
\end{proposition}
\begin{proof}
  Suppose $s=t \circ f$.
  Set $g=(\id[S] \otimes f) \circ (\id[S] \otimes s)^{-1} \circ \rho_S^{-1} \colon S \to S \otimes T$. Then
  \[
    \rho_S \circ (\id[S] \otimes t) \circ g
    = \rho_S \circ (\id[S] \otimes s) \circ (\id[S] \otimes s)^{-1} \circ {\rho_S}^{-1} 
    = \id[S]\text.
  \]
  Idempotence of $t$ makes $\id[S \otimes T] \otimes t \colon (S \otimes T) \otimes T \to T$ an isomorphism. Hence, by Lemma~\ref{lem:splitmonosubunits}, so is $\id[S] \otimes t$. 

  Conversely, suppose $\id[S] \otimes t$ is an isomorphism. Because the following diagram  commutes
  \[\begin{tikzcd}
    S \otimes T \dar{\id[S] \otimes t} \rar{s \otimes \id[T]}
    & I \otimes T \dar{\id[I] \otimes t} \rar{\rho_T}
    & T \dar{t} \\
    S \otimes I \rar{s \otimes \id[I]}
    & I \otimes I \rar{\rho_I}
    & I
  \end{tikzcd}\]
  the bottom row $s$ factors through the right vertical arrow $t$.
\end{proof}

\begin{corollary} 
  \label{cor:subunitdomain}
  If $s, s' \colon S \rightarrowtail I$ are idempotent subunits, $s' = s \circ f$ for a unique $f$, which is an isomorphism.
\end{corollary}
\begin{proof}
  This follows from Proposition~\ref{prop:subunits_less} because $\id[S] \otimes s'$ and $s \otimes \id[S]$ are isomorphisms.
\end{proof}

\begin{proposition}
  The idempotent subunits in a firm braided monoidal category form an idempotent commutative monoid under $\otimes$ with unit $\id[I]$.
  \[
    \big(s \colon S \rightarrowtail I\big) \otimes \big(t \colon T \rightarrowtail I\big) = \big(\lambda_I \circ (s \otimes t) \colon S \otimes T \to I\big)
  \]
\end{proposition}
\begin{proof}
  First observe that $s \otimes t = (\id[I] \otimes t) \circ (s \otimes \id[T])$ is indeed a subunit, because $\id[I] \otimes t = \lambda_I^{-1} \circ t \circ \lambda_T$ is monic, and $s \otimes \id[T]$ is monic by firmness.
  It is easily seen to be idempotent using the braiding.

  Next, $\id[I] \otimes s = \lambda_I \circ (\id[I] \otimes s) = s \circ \lambda_S$  represents the same subobject as $s$. Similarly $\id[I] \otimes s$ represents the same subobject as $s$ because $\rho_I=\lambda_I$.
  An analogous argument using coherence establishes associativity. For commutativity, use the braiding $\sigma_{S,T}$ to observe that $s \otimes t$ and $t \otimes s$ represent the same subobject. For idempotence note that $s \otimes s$ and $s$ represent the same subobject because $\lambda_I \circ (s \otimes s) = s \circ \rho_S \circ (\id[S] \otimes s)$.
\end{proof}

\begin{corollary}
  The idempotent subunits in a firm braided monoidal category form a meet-semilattice under the usual order of subobjects, with meet given by $\otimes$, and largest element $\id[I]$.
  \[
  \begin{pic}[xscale=2,yscale=1.5]
    \node (s) at (0,0) {$S$};
    \node (i) at (1,.5) {$I$};
    \node (t) at (0,1) {$T$};
    \draw[>->] (s) to node[below]{$s$} (i);
    \draw[>->] (t) to node[above]{$t$} (i);
    \draw[>->,dashed] (s) to (t);
  \end{pic}
  \qquad \iff \qquad
  \begin{pic}[xscale=2,yscale=1.5]
    \node (s) at (0,0) {$S$};
    \node (i) at (1,0) {$I$};
    \node (st) at (0,1) {$S \otimes T$};
    \node (ii) at (1,1) {$I \otimes I$};
    \draw[>->] (s) to node[below]{$s$} (i);   
    \draw[>->] (st) to node[above]{$s \otimes t$} (ii);
    \draw[->] (ii) to node[right]{$\lambda_I$} node[left]{$\simeq$} (i);
    \draw[->,dashed] (s) to node[left]{$\simeq$} (st);
  \end{pic}
  \]  
\end{corollary}
\begin{proof}
  If $s$ and $s \otimes t$ represent the same subobject, then $S \simeq S \otimes T$, making $\id[S] \otimes t$ an isomorphism and so $s \leq t$ by Proposition~\ref{prop:subunits_less}. 
  Conversely, if $s \leq t$ then by the same proposition $\id[S] \otimes t$ is an isomorphism with $s = \lambda_I \circ (s \otimes t) \circ (\id[S] \otimes t)^{-1} \otimes \rho_S^{-1}$, and so both subobjects are equal. 
\end{proof}

\begin{example} \label{ex:idemcmonascat}
  Any meet-semilattice $(L, \wedge, 1)$ forms a strict symmetric monoidal category: objects are $x \in L$, there is a unique morphism $x \to y$ if $x \leq y$, tensor product is given by meet, and tensor unit is $I = 1$. Every morphism is monic so this monoidal category is firm, and its (idempotent) subunits are $(L, \wedge, 1)$.

  This gives the free firm symmetric monoidal category on a meet-semilattice. More precisely: this construction is a functor from the category of meet-semilattices and their homomorphisms to the category of firm braided monoidal categories with (strong) monoidal functors that preserve subunits; moreover it is left adjoint to the functor that takes idempotent subunits.
  
\end{example}

\section{Restriction and localisation}

If $U$ is an open subset of a locally compact Hausdorff space $X$, then any Hilbert $C_0(X)$-module induces a Hilbert $C_0(U)$-module. This section shows that this restriction behaves well in any monoidal category.

\begin{definition}
	Let $s$ be an idempotent subunit in a monoidal category $\cat{C}$. Define the \emph{restriction} of $\cat{C}$ to $s$, denoted by $\cat{C}\restrict{s}$, to be the full subcategory of $\cat{C}$ of objects $E$ for which $\id[E] \otimes s$ is an isomorphism.
\end{definition}

\begin{proposition}\label{prop:restrictioncoreflective}
	If $s$ is an idempotent subunit in a monoidal category $\cat{C}$, then $\cat{C}\restrict{s}$ is a coreflective monoidal subcategory of $\cat{C}$.
  \[\begin{pic}
    \node (l) at (0,0) {$\cat{C}$};
    \node (r) at (3,0) {$\cat{C}\restrict{s}$};
    \node at (1.5,0) {$\top$};
    \draw[->] ([yshift=1mm]l.east) to[out=10,in=170] ([yshift=1mm]r.west);
    \draw[left hook->] ([yshift=-1mm]r.west) to[out=-170,in=-10] ([yshift=-1mm]l.east);
  \end{pic}\]
  The right adjoint $\cat{C} \to \cat{C}\restrict{s}$, given by $E \mapsto E \otimes S$ and $f \mapsto f \otimes \id[S]$, is also called \emph{restriction} to $s$.
\end{proposition}
\begin{proof}
	First, if $E \in \cat{C}$, note that $E \otimes S$ is indeed in $\cat{C}\restrict{s}$ because $\id[E \otimes S] \otimes s = \alpha^{-1} \circ (\id[E] \otimes (\id[S] \otimes s)) \circ \alpha$ and $s$ is idempotent. 
  Similarly, $\cat{C}\restrict{s}$ is a monoidal subcategory of $\cat{C}$. 
  Finally, there is a natural bijection 
	\begin{align*}
		\cat{C}(E,F) &\simeq  \cat{C}\restrict{s}(E,F\otimes S)\\
		f &\mapsto (f\otimes\id[s]) \circ (\id[E]\otimes s)^{-1} \circ \rho_E^{-1}\\
		\rho_F\circ (\id[F]\otimes s) \circ g &\mapsfrom g
	\end{align*}
  for $E \in \cat{C}\restrict{s}$ and $F \in \cat{C}$.
  So restriction is right adjoint to inclusion.
  For monoidality, see~\cite[Thm~5]{jacobsmandemaker:coreflections}.
\end{proof}

\begin{example}
  Restricting $\cat{Hilb}_{C_0(X)}$ to the idempotent subunit~\eqref{eq:subunithilbertmodules} induced by an open subset $U \subseteq X$ does nearly, but not quite, give $\cat{Hilb}_{C_0(U)}$: one gets the full subcategory of Hilbert $C_0(X)$-modules $E$ for which $\|x\|(X \setminus U)=0$ for all $x \in E$. Any such Hilbert $C_0(X)$-module also forms a $C_0(U)$-module. But conversely there is no obvious way to extend the action of scalars on a general $C_0(U)$-module to make it a $C_0(X)$-module. 
  There is a so-called \emph{local} adjunction between $\cat{Hilb}_{C_0(X)}\restrict{C_0(U)}$ and $\cat{Hilb}_{C_0(U)}$, which is only an adjunction when $U$ is clopen~\cite[Proposition~4.3]{clarecrisphigson:adjoint}.
  This problem does not occur in a purely algebraic setting, see Appendix~\ref{sec:firmrings}.
\end{example}

\begin{proof}
  We first prove that $E \in \cat{C}\restrict{S}$ if and only if $|x|(X\setminus U)=0$ for all $x \in E$, where $|x|^2 = \langle x, x \rangle_{C_0(X)}$.
  On the one hand, if $x \in E$ and $f \in S$ then $|x \otimes f|(X \setminus U) = |x| |f| (X \setminus U)=0$. 
  Therefore $|x|(X \setminus U)=0$ for all $x \in E \otimes S$.
  Because $E \otimes S \simeq E$ is an isomorphism, $|x|(X \setminus U)=0$ for all $x \in E$.

  On the other hand, suppose that $|x|(X \setminus U)=0$ for all $x \in E$.
  We are to show that the morphism $E \otimes S \to E$ given by $x \otimes f \mapsto xf$ is bijective.
  To see injectivity, let $f \in S$ and $x \in E$, and suppose that $xf=0$. Then $|x||f|=|xf|=0$, so for all $t \in U$ either $|x|(t)=0$ or $f(t)=0$. So $|x \otimes f|(U)=0$, and hence $x \otimes f=0$.
  To see surjectivity, let $x \in E$. Then $|x|(t)=0$ for all $t \in X \setminus U$. So $x=\lim xf_n$ for an approximate unit $f_n$ of $S$. But that means $x$ is the image of $\lim x \otimes f_n$.
\end{proof}

Localisation in algebra generally refers to a process that adds formal inverses to an algebraic structure~\cite[Chapter 7]{kashiwara2005categories}. We will show that restriction is an example of localisation in this sense.

\begin{definition}
	Let $\cat{C}$ be a category and $\Sigma$ a collection of morphisms in $\cat{C}$. 
  A \emph{localisation of $\cat{C}$ at $\Sigma$} is a category $\cat{C}[\Sigma^{-1}]$ together with a functor $Q \colon \cat{C} \to \cat{C}[\Sigma^{-1}]$ such that:
	\begin{itemize}
		\item $Q(f)$ is an isomorphism for every $f\in \Sigma$;
		\item for any functor $R\colon\cat{C} \rightarrow \cat{D}$ such that $R(f)$ is an isomorphism for all $f\in\Sigma$, there exists a functor $\overline{R}\colon\cat{C}[\Sigma^{-1}] \rightarrow \cat{D}$ and a natural isomorphism $\overline{R}\circ Q \simeq R$;
    \[\begin{tikzpicture}
      \node (l) at (0,1.3) {$\cat{C}$};
      \node (tr) at (3,1.3) {$\cat{C}[\Sigma^{-1}]$};
      \node (br) at (3,0) {$\cat{D}$};
      \draw[->] (l) to node[above]{$Q$} (tr);
      \draw[->] (l) to node[below]{$R$} (br);
      \draw[->,dashed] (tr) to (br);
      \node at (2,.9) {$\simeq$};
    \end{tikzpicture}\]
		\item precomposition  
		$
			(-)\circ Q \colon \big[\cat{C}[\Sigma^{-1}], \cat{D}\big] \to [\cat{C}, \cat{D}]
		$
		is full and faithful for every category $\cat{D}$.
	\end{itemize}
\end{definition}

\begin{proposition}
	Restriction $\cat{C} \to \cat{C}\restrict{s}$ at an idempotent subunit $s$
  is a localisation of $\cat{C}$ at $\{ \id[E]\otimes s \mid E\in\cat{C} \}$.
\end{proposition}
\begin{proof}
  Observe that $(-)\otimes S$ sends elements of $\Sigma$ to isomorphisms because $s$ is idempotent. 
	Let $R\colon\cat{C}\rightarrow\cat{D}$ be any functor making $R(\id[E] \otimes s)$ an isomorphism for all $E\in\cat{C}$. 
  Define $\overline{R}\colon \cat{C}\restrict{s} \to \cat{D}$ by $E \mapsto R(E)$ and $f \mapsto R(f)$.
  Then 
	\begin{equation*}
		\eta_E = R(\rho_E)\circ R(\id[E]\otimes s) \colon R(E\otimes S) \rightarrow R(E)
	\end{equation*}
  is a natural isomorphism.
	It is easy to check that precomposition with restriction is full and faithful.
\end{proof}

Above we restricted along one individual idempotent subunit $s$. Next we investigate the structure of the family of these functors when $s$ varies, which we simply call \emph{restriction}.

\begin{definition} \label{def:graded_monad} \cite{fujiikatsumatamellies:gradedmonads}
  Let $\cat{C}$ be a category and $(\cat{E}, \otimes, 1)$ a monoidal category. Denote by $[\cat{C}, \cat{C}]$ the monoidal category of endofunctors of $\cat{C}$ with $F\otimes G = G\circ F$.  An \emph{$\cat{E}$-graded monad} on $\cat{C}$ is a monoidal functor $T\colon \cat{E} \rightarrow [\cat{C}, \cat{C}]$. More concretely, an $\cat{E}$-graded monad consists of:
  \begin{itemize}
    \item a functor $T\colon \cat{E} \rightarrow [\cat{C}, \cat{C}]$;
    \item a natural transformation $\eta \colon \id[\cat{C}] \Rightarrow T(1)$;
    \item a natural transformation $\mu_{s,t}\colon T(t)\circ T(s) \rightarrow T(s\otimes t)$ for every two objects $s,t$ in $\cat{E}$;
  \end{itemize}
  making the following diagrams commute for all $r,s,t$ in $\cat{E}$.
  \[\begin{pic}[xscale=2.5, yscale=1.5]
  \node (TsTtTu) at (1,2) {$T(t)\circ T(s)\circ T(r)$};
  \node (TstTu) at (0,1) {$T(t)\circ T(r\otimes s)$};
  \node (Tstu1) at (0,0) {$T((r\otimes s)\otimes t)$};
  \node (Tstu2) at (2,0) {$T(r\otimes (s\otimes t))$};
  \node (TsTtu) at (2,1) {$T(t\otimes s)\circ T(r)$};
  \draw[->] (TsTtTu) to node[left]{$\mu_{r,s}\otimes \id[T(t)]$} (TstTu);
  \draw[->] (TstTu) to node[left]{$\mu_{r\otimes s,t}$} (Tstu1);
  \draw[->] (Tstu1) to node[below]{$T(\alpha_{r,s,t})$} (Tstu2);
  \draw[->] (TsTtu) to node[right]{$\mu_{r,s\otimes t}$} (Tstu2);
  \draw[->] (TsTtTu) to node[right]{$\id[T(r)]\otimes\mu_{s,t}$} (TsTtu); 
  \end{pic}
  \qquad
  \begin{pic}[xscale=2.3,yscale=.8]
  \node (idc Ts) at (0,1.5) {$T(s)\circ \id[\cat{C}]$};
  \node (Ts) at (0,0) {$T(s)$};
  \node (T1s) at (1.5,0) {$T(1\otimes s)$};
  \node (T1Ts) at (1.5,1.5) {$T(s)\circ T(1)$};
  \draw[-, double] (idc Ts) to node[left]{} (Ts);
  \draw[->] (idc Ts) to node[below]{$\eta\otimes \id[T(s)]$} (T1Ts);
  \draw[->] (T1Ts) to node[right]{$\mu_{1,s}$} (T1s);
  \draw[->] (T1s) to node[below]{$T(\lambda_s)$} (Ts);
  \begin{scope}[yshift=23mm]
  \node (Ts idc) at (0,1.5) {$\id[\cat{C}]\circ T(s)$};
  \node (Ts) at (0,0) {$T(s)$};
  \node (Ts1) at (1.5,0) {$T(s\otimes 1)$};
  \node (TsT1) at (1.5,1.5) {$T(1)\circ T(s)$};
  \draw[-, double] (Ts idc) to node[left]{} (Ts);
  \draw[->] (Ts idc) to node[above]{$\id[T(s)]\otimes \eta$} (TsT1);
  \draw[->] (TsT1) to node[right]{$\mu_{s,1}$} (Ts1);
  \draw[->] (Ts1) to node[above]{$T(\rho_s)$} (Ts);
  \end{scope}
  \end{pic}\]
\end{definition}

\begin{theorem}
  Let $\cat{C}$ be a monoidal category.
  Restriction is a monad graded over the idempotent subunits, where we do not identify monomorphisms representing the same subobject: $\cat{E}$ has as objects monomorphisms $s \colon S \rightarrowtail I$ in $\cat{C}$ with $s \otimes \id[S]$ an isomorphism, and as morphisms $f \colon s \to t$ those $f$ with $s = t \circ f$.
\end{theorem}
\begin{proof}
  The functor $\cat{E} \to [\cat{C},\cat{C}]$ sends $s \colon S \rightarrowtail I$ to $(-) \otimes S$, and $f$ to the natural transformation $\id[(-)] \otimes f$.
  The natural transformation $\eta_E \colon E \to E \otimes I$ is given by $\rho_E^{-1}$.
  The family of natural transformations $\mu_{s,t} \colon ((-)\otimes S)\otimes T \rightarrow (-)\otimes(S\otimes T)$ is given by $\alpha_{(-),S,T}$. Associativity and unitality diagrams follow.
\end{proof}

\section{Support}

So far we have focused on the spatial structure encoded within the tensor unit. This section investigates how this spatial structure influences (morphisms between) arbitrary objects in a monoidal category.

\begin{definition} \label{def:supportin}
  A morphism $f \colon E \to F$ has \emph{support in} a morphism $s \colon S \to I$ when it factors through the morphism $\rho_F \circ (\id[F] \otimes s)$. We will particularly be interested in the case when $s$ is a subunit.

\end{definition}

\begin{proposition} \label{prop:local}
  The following are equivalent for an object $E$ and idempotent subunit $s \colon S \rightarrowtail I$:
  \begin{enumerate}[label=(\alph*)]
    \item \label{cond:local_canon}
      $\id[E] \otimes s \colon E \otimes S \to E \otimes I$ is an isomorphism;
    \item \label{cond:local_some}
      there is an isomorphism $E \otimes S \simeq E$; 
    \item \label{cond:local_someob}
      there is an isomorphism $E \simeq F \otimes S$ for some object $F$;
    \item \label{cond:supported}
      $\id[E]\colon E \to E$ has support in $s$.
  \end{enumerate}
\end{proposition}
\begin{proof}
  Trivially \ref{cond:local_canon} $\implies$ \ref{cond:local_some} $\implies$ \ref{cond:local_someob}. For \ref{cond:local_someob}$\implies$\ref{cond:supported}:
  by idempotence $\id[F \otimes S] \otimes s$ is an isomorphism, 
  so if $E \simeq F \otimes S$ then also $\id[E]$ is an isomorphism by Lemma~\ref{lem:splitmonosubunits}.
  For \ref{cond:supported}$\implies$\ref{cond:local_canon}: if $\id[E]$ factors through $\id[E] \otimes s$, then by idempotence $\id[E \otimes S] \otimes s$ is an isomorphism, and hence so is $\id[E] \otimes s$ by Lemma~\ref{lem:splitmonosubunits}.
\end{proof}

Now comes the main result: a simple observation, but the basis for the application in section~\ref{sec:teleportation}.

\begin{lemma} \label{lem:supportin}
  Let $s \colon S \to I$ and $t \colon T \to I$ be morphisms in a firm monoidal category.
  \begin{enumerate}[label=(\alph*)]
    \item If $f \colon E \to F$ has support in $s$, and $g \colon F \to G$ has support in $t$, then $g \circ f$ has support in $t \otimes s$.
    \item If $f$ has support in $s$, and $g$ has support in $t$, then $g \otimes f$ has support in $t \otimes s$.
  \end{enumerate}
\end{lemma}
\begin{proof}
  Straightforward diagram chase. 
\end{proof}

In particular, if $\id[E]$ or $\id[F]$ has support in an idempotent subunit $s$, then so does any map $E \to F$. Similarly, if $f$ or $g$ has support in $s$, so does $g \circ f$.
More generally, in a firm braided monoidal category, any morphism built from a finite number of maps $f_i$ with support in $s_i$ using $\otimes$ and $\circ$ has support in $\bigwedge s_i$.

\begin{remark}
  In general braided monoidal categories, we can only say that morphisms have support \emph{within} some idempotent subunit. In $\cat{Hilb}_{C_0(X)}$, the idempotent subunits form a \emph{complete} meet-semilattice: there is a greatest lower bound (and a least upper bound) of any family of idempotent subunits. In such categories we can speak of \emph{the} support of a morphism $f$ as the largest idempotent subunit $\supp(f)$ that it restricts to. Hence the previous lemma says $\supp(g \circ f) \leq \supp(g) \wedge \supp(f)$ and $\supp(f \otimes g) \leq \supp(f) \wedge \supp(g)$. In $\cat{Hilb}_{C_0(X)}$, supports moreover respect the tensor: $\supp(f) \otimes \supp(g) = \supp(f \otimes g)$.

  We leave open the question of what structure on the firm braided monoidal category $\cat{C}$ might ensure that the idempotent subunits form a distributive lattice, or even a locale, as in $\cat{Hilb}_{C_0(X)}$. One answer is the following. It is straightforward to see that if $\cat{C}$ has distributive coproducts and supports respecting the tensor in the above sense, the idempotent subunits form a lattice via $s \vee t = \supp([s,t])$.
\end{remark}

\section{Causal structure}

In Section~\ref{sec:hilbertmodules} we regarded $\cat{Hilb}_{C_0(X)}$ as a naive quantum field theory. 
What if $X$ has more structure than just a (topological) space? Can we model spacetime structure of $X$ using only categorical properties of $\cat{Hilb}_{C_0(X)}$? The first evidence is affirmative, as this section discusses.

\begin{remark}
  Causality relations may be treated formally when $X$ is a Lorentzian manifold with a time orientation, also known as a \emph{spacetime}~\cite{penrose1972techniques}. For points $s,t \in X$, write $s \chronprec t$ when there is a future-directed timelike curve from $s$ to $t$, or more generally $s \causprec t$  when there is a future-directed non-spacelike such curve. Both relations are  transitive, and if $s \chronprec t$ then $s \causprec t$. This defines four sets for each $t \in X$:
  \begin{center}\begin{tabular}{r|cc}
   & chronological & causal \\
   \hline
   future 
   & $I^+(t) = \{ s \in X \mid t \chronprec s\}$
   & $J^+(t) = \{ s \in X \mid t \causprec s\}$ \\
   past 
   & $I^-(t) = \{ s \in X \mid s \chronprec t \}$
   & $J^-(t) = \{ s \in X \mid s \causprec t \}$
  \end{tabular}\end{center}
  Conversely, the manifold structure of $X$ may be reconstructed from the relation $\causprec$ under the mild physical condition of global hyperbolicity using domain-theoretic techniques~\cite{martin2006domain}.
\end{remark}

Extend these definitions to subsets $S \subseteq X$ by $I^{\pm}(S) = \bigcup_{s \in S} I^{\pm}(s)$ and $J^{\pm}(S) = \bigcup_{s \in S} J^{\pm}(s)$. The subset $I^{+}(S)$ of $X$ is open, with $I^{+}(I^{+}(S)) = I^{+}(S) \subseteq J^{+}(S) = J^{+}(J^{+}(S)) \subseteq \closure{I^{+}(S)}$, and the same properties hold for $I^-, J^-$~\cite{penrose1972techniques}.
In fact, when $S$ is itself open the chronological and causal futures (and pasts) coincide. 

\begin{proposition} \label{prop:closurescoincide}
  If $S$ is an open subset of a spacetime $X$, then $S \subseteq I^\pm(S)$ and $I^\pm(S) = J^\pm(S)$.
\end{proposition}
\begin{proof}
  We discuss $I^+$ and $J^+$, the other case is similar.
  For the first statement we need to show that for all $t \in S$ there is $s \in S$ with $s \chronprec t$. By definition of $\causprec$ always $t \in J^-(t)$ and so $t \in \closure{I^-(t)}$. Hence any open neighbourhood of $t$ has non-empty intersection with $I^-(t)$. In particular, there exists $s \in S \cap I^-(y)$.

  Next we show that $S \subseteq I^+(S)$ implies $I^+(S) = J^+(S)$.
  Let $t \in J^+(S)$, say $s \causprec t$ for $s \in S$. By the above there is $r \in S$ with $r \chronprec s$. Hence $r \chronprec t$, and so $t \in I^+(S)$. Conversely,  $I^+(S) \subseteq J^+(S)$ always holds.
\end{proof}

Thus we may view $I^+(-)$ and $I^-(-)$ as the `future' and `past' operators, providing a causal structure on the idempotent subunits of $\cat{Hilb}_{C_0(X)}$. 
Open subsets of $X$ closed under $I^+$ and $I^-$ are often called \emph{future sets} and \emph{past sets}~\cite[Section~3]{penrose1972techniques}. 

Let us generalise this to arbitrary monoidal categories.

\begin{definition}\label{def:closureoperator}
  A \emph{closure operator} on a partially ordered set $P$ is a function $C \colon P \to P$ satisfying:
  \begin{itemize}
    \item if $s \leq t$, then $C(s) \leq C(t)$;
    \item $s \leq C(s)$;
    \item $C(C(s)) \leq C(s)$.
  \end{itemize} 
  An element $s \in P$ is \emph{$C$-closed} when $s = C(s)$.
  A \emph{causal structure} on a monoidal category consists of a pair $(C^+, C^-)$ of closure operators on its partially ordered set of idempotent subunits.
\end{definition}



\begin{proposition}\label{prop:causalstructurerestricts}
  Causal structure restricts: if $r$ is an idempotent subunit in a firm braided monoidal category $\cat{C}$, and $C$ a closure operator, then $D(s) = C(s) \wedge r$ is a closure operator on $\cat{C}\restrict{r}$.
\end{proposition}
\begin{proof}
  The idempotent subunits in $\cat{C}\restrict{r}$ are those subunits $s$ in $\cat{C}$ with $s \leq r$ by Proposition~\ref{prop:restrictioncoreflective}. 
  If $s \leq t \leq r$, then $C(s)\leq C(t)$, so $D(s)\leq D(t)$.
  If $s \leq r$, then $s \leq C(s)$, and hence also $s \leq D(s)$.
  Finally, if $s \leq r$, then $C(s) \wedge r \leq C(s)$, so $C(s) \leq C(C(s)) \leq C(C(s) \wedge r)$, and so $D(D(s)) \leq D(s)$.
\end{proof}

\section{Teleportation} \label{sec:teleportation}

This section models the quantum teleportation protocol as in~\cite{abramskycoecke:categoricalsemantics} using idempotent subunits. As a proof of concept, we show that it is only successfully supported on the intersection of Alice and Bob's future sets. 
This demonstrates how one may reason about spatial aspects of protocols, without needing tools outside monoidal categories~\cite{blutecomeau:neumann}.
Fix a firm braided monoidal category $\cat{C}$ with a closure operator $C^+$ on the idempotent subunits.
Intuitively, think of the following `spacetime' diagram.
    
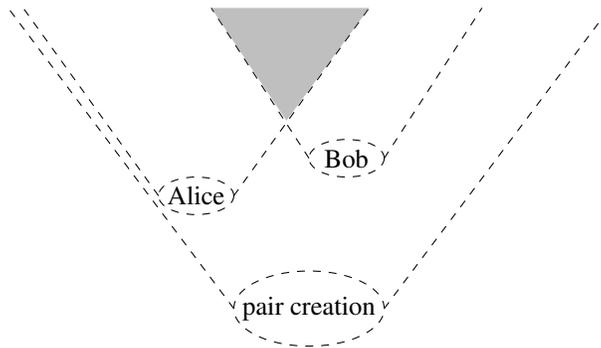
\[\begin{tikzpicture}
      \draw[dashed] (0,0) ellipse (1 and .5);
      \draw[dashed] (1,0) to (4,4);
      \draw[dashed] (-1,0) to (-4,4);
        \node[font=\small] at (0,0) {pair creation};
        \draw[dashed] (-1.5,1.5) ellipse (.5 and .25);
        \node[font=\small] at (-1.5,1.5) {Alice};
        \draw[dashed] (-2,1.5) to (-3.8,4);
        \draw[dashed] (-1,1.5) to (.8,4);
      \begin{scope}[xshift=20mm,yshift=5mm]
        \draw[dashed] (-1.5,1.5) ellipse (.5 and .25);
        \node[font=\small] at (-1.5,1.5) {Bob};
        \draw[dashed] (-2,1.5) to (-3.3,3.5);
        \draw[dashed] (-1,1.5) to (.3,3.5);
        \draw[draw=none, fill=gray, opacity=.5] (-2.3,2) -- (-3.3,3.5) -- (-1.2,3.5) -- cycle;
      \end{scope}
    \end{tikzpicture}\]

The first step is to generate an entangled pair of particles. Say this happens in some laboratory, whose location is modeled by an idempotent subunit $r$.  All that follows in the protocol happens in the future set $C^+(r)$ of $r$. Without loss of generality we may assume that $\cat{C} = \cat{C}\restrict{C^+(r)}$, or in other words, that $r=\id[I]$. The pair creation is represented by a state $\eta \colon I \to A' \otimes B$. We think of $A'$ and $B$ as two fields, aiming to model teleportation into $B$ from some other field $A$, which Alice will interact with via $A'$. 

The protocol begins by sending Alice and Bob each their half of the entangled pair $\eta$. 
Alice and Bob both have laboratories, that we think of as the open regions in the diagram. Formally these are idempotent subunits $s$ and $t$. Their future sets, drawn as the upward cones in the diagram, are $C^+(s)$ and $C^+(t)$.
Since we think of idempotent subunits as spacetime regions,
\begin{center}
  restriction = propagation.  
\end{center}
That is, sending $A'$ to Alice means restricting it to Alice's future, replacing it with $A' \otimes C^+(s)$. Similarly, sending $B$ to Bob means restricting $B$ to $C^+(t)$. Hence we replace $\eta$ with $\eta' = \eta \otimes \id[C^+(s)] \otimes \id[C^+(t)]$.

Next, Alice receives the unknown input qubit in her laboratory, performs her measurement, and postselects. Since this occurs within her laboratory, it restricts to $C^+(s)$, giving a map $\varepsilon \colon A \otimes A' \otimes C^+(s) \to I$. 
In the probabilistic protocol, Bob receives Alice's state with some probability. We could use Frobenius structures to model Alice's classical communication of her outcome to Bob~\cite{heunenreyes:frobenius}, but this doesn't change the outcome of the story at all.

The whole protocol, omitting coherence isomorphisms, is thus described by the morphism: 
\begin{equation} \label{eq:tele-morphism}
  (\varepsilon \otimes \id[B \otimes C^+(t)]) \circ (\id[A] \otimes \eta') \colon A \otimes C^+(s) \otimes C^+(t) \to B \otimes C^+(t)\text.
\end{equation}
To be successful, it must restrict to some known isomorphism between its domain and $B \otimes C^+(s) \otimes C^+(t)$, \textit{i.e.}\ within Alice and Bob's common future. 
For example, it is easy to see that the unit $\eta$ and counit $\varepsilon$ of a pair of dual objects provide this, after tensoring them with $\id[C^+(s)]$ and $\id[C^+(t)]$ as appropriate (with the induced isomorphism being the identity). In $\cat{Hilb}_{C_0(X)}$, we can use Example~\ref{ex:constanthilbertmodule} for this choice.

All in all, Lemma~\ref{lem:supportin} now says that, in any case, the total morphism~\eqref{eq:tele-morphism} modelling the protocol always has support in $C^+(s) \otimes C^+(t)$, the intersection of Alice and Bob's futures. In general it is not an isomorphism, but only one when restricted to this intersection.

\section{Open questions}

This article is the first step in a larger programme of studying causality using only tools from monoidal categories. There are many open questions for future investigation.
\begin{itemize}
  \item Can we incorporate support and restriction into the graphical calculus for monoidal categories?
  \item What categorical properties ensure that idempotent subunits form a locale, as in $\cat{Hilb}_{C_0(X)}$? Can $\cat{Hilb}_{C_0(X)}$ be regarded as the category of Hilbert spaces internal to some ambient category, whose categorical logic governs supports?
  \item How do idempotent subunits connect to the deeper categorical theory of self-similarity~\cite{hines:coherenceselfsimilarity,fioreleinster:thompsonsgroupf}?
  \item Applications to quantum foundations remain to be explored. How do superselection sectors relate to complemented idempotent subunits, or Bell inequalities to the support of an entangled pair?
  \item How does our work relate to approaches to causality using discarding maps~\cite{coecke2013causal,kissingeruijlen:causal}? One might hope to establish a connection using the 
  CP*-construction~\cite{coeckeheunenkissinger:cpstar,heunenreyes:frobenius} for dagger categories.
  Though $\cat{Hilb}_{C_0(X)}$ is not a dagger category, its so-called \emph{adjointable} morphisms form a dagger subcategory~\cite{heunenreyes:frobenius}. However, considering only these appears to lose spatial structure, since a subunit $C_0(U) \to C_0(X)$ is only adjointable when $U$ is clopen in $X$~\cite[Lemma~3.3]{heunenreyes:frobenius}.  
  \item Is there a structure on the whole category inducing the causal structure on the idempotent subunits? Aside from these, what other interesting structure on $X$ might transfer to structure on $\cat{Hilb}_{C_0(X)}$?
  \item How can this framework model other tasks from relativistic quantum information theory, such as position-based quantum cryptography~\cite{haydenmay:wherewhen}, or summoning~\cite{kent:tasks}?
\end{itemize}

The last two questions ask for relating $C^+$ and $C^-$. We end with a preliminary partial answer. 
In any monoidal category with a zero object $0$ satisfying $E \otimes 0 \simeq 0 \simeq 0 \otimes E$ for all objects $E$, there is a least idempotent subunit $0 \colon 0 \rightarrowtail E$. Call subobjects $s, t$ \emph{disjoint} when $r \leq s$ and $r \leq t$ imply $r = 0$.

\begin{proposition} 
  For any spacetime $X$, the causal structure $I^\pm$ on $\cat{Hilb}_{C_0(X)}$ has \emph{complements}: for every $I^+$-closed subunit $s$ there is a unique $I^-$-closed subunit $t$ disjoint from $s$ such that a subunit $r$ is $0$ whenever it is disjoint from $s$ and $t$, and the same holds when interchanging $I^+$ and $I^-$. 
\end{proposition}
\begin{proof}

 Let $F$ be a future set. We need to show there is a unique past set $P$ with $F \wedge P = \emptyset$ and $X = \closure{F} \vee \closure{P}$. We first show that whenever this holds, $F$ and $P$ have the same boundary. 
  Note that $\closure{F} = \{t \mid I^+(t) \subseteq F\}$ and $\closure{P} = \{t \mid I^-(t) \subseteq P\}$~\cite[3.4]{penrose1972techniques}. Suppose $t \in \closure{F} \setminus F$; we will show that $t \in \closure{P} \setminus P$, and the converse holds similarly.
  Since $X = \closure{F} \vee \closure{P}$, for all $s$ either $I^+(s) \subseteq F$ or $I^-(s) \subseteq P$. So if $s \in I^-(t)$ we must have $I^-(s) \subseteq P$, because otherwise $t \in F$ gives a contradiction. 
  Therefore $I^-(t) \subseteq P$, and so $t \in \closure{P}$. 
  But if $t$ is in the past set $P$, then there is $s \in P$ with $t \chronprec s$, contradicting $I^+(t) \subseteq F \subseteq X \setminus P$. 
  Thus $F$ and $P$ have the same boundary $B$, and $X$ is a disjoint union $F \cup B \cup P$. Hence $P=X \setminus \closure{F}$ is unique. 
  Conversely, for any future set $F$, $P = X \setminus \closure{F}$ is easily seen to be such a past set, using the expression for $\closure{F}$ above.
\end{proof}

\bibliographystyle{eptcs}
\bibliography{localisation}

\appendix
\section{Firm rings}\label{sec:firmrings}

This appendix discusses the purely algebraic example of modules over rings that are in general nonunital, but have local units, stripping away the analytic details of Hilbert modules.

\begin{definition}
  A ring $R$ is \emph{firm} when its multiplication is a bijection $R \otimes R \to R$, and \emph{nondegenerate} when $r \in R$ vanishes as soon as $rs=0$ for all $s \in R$.
  Any unital ring is firm and nondegenerate, but examples also include 
  infinite direct sums $\bigoplus_{n \in \mathbb{N}} R_n$ of unital rings $R_n$.
  Firm rings $R$ are \emph{idempotent}: they equal $R^2=\{\sum_{i=1}^n r_i' r_i'' \mid r_i',r_i'' \in R\}$. 
  Let $R$ be a nondegenerate firm commutative ring. An $R$-module $E$ is \emph{firm} when the scalar multiplication is a bijection $E \otimes R \to E$~\cite{quillen:nonunitalrings}, and \emph{nondegenerate} when $x \in E$ vanishes as soon as $xr=0$ for all $r \in R$.
  If $R$ is unital, then every $R$-module is firm and nondegenerate.
  Nondegenerate firm $R$-modules and linear maps form a monoidal category $\cat{FMod}_R$.
\end{definition}

\begin{example}
  The idempotent subunits in $\cat{FMod}_R$ correspond to \emph{nondegenerate firm idempotent ideals}: ideals $S \subseteq R$ that are idempotent as rings, and nondegenerate and firm as $R$-modules. 
  Any ideal that is unital as a ring is a nondegenerate firm idempotent ideal. 
  The category $\cat{FMod}_R$ is firm.
\end{example}
\begin{proof}
  Monomorphisms are injective by nondegeneracy, so every idempotent subunit is a nondegenerate firm $R$-submodule of $R$, that is, a nondegenerate firm ideal. Because the inclusion $S \otimes S \to R \otimes S$ is surjective and $S$ is firm, the map $S \otimes S \to S$ given by $s' \otimes s''\mapsto s's''$ is surjective. Thus $S$ is idempotent.
  
  Conversely, let $S$ be a nondegenerate firm idempotent ideal of $R$. The inclusion $S \otimes S \to R \otimes S$ is surjective, as $r \otimes s \in R \otimes S$ can be written as $r \otimes s's'' = r s' \otimes s'' \in S \otimes S$. Hence $S$ is an idempotent subunit.
  
  Next suppose ideal $S$ is unital (with generally $1_S \neq 1_R$ if $R$ is unital).
  Then $S \otimes R \to S$ given by $s \otimes r \mapsto sr$ is bijective: surjective as $1_S \otimes s \mapsto 1_S s = s$; and injective as $s \otimes r = 1_S \otimes sr = 1_S \otimes 0 = 0$ if $sr=0$.
  Hence $S$ is firm and nondegenerate.
  Any $s \in S$ can be written as $s=s 1_S \in S^2$, so $S$ is idempotent.

  Finally, to see that the category is firm, let $S,T \subseteq R$ be nondegenerate firm idempotent ideals.
  We need to show that the map $S \otimes T \mapsto R \otimes T$ given by $s \otimes t \mapsto s \otimes t$ is injective.
  Because $T$ is firm, it suffices to show that the map $S \otimes T$ given by $s \otimes t \mapsto st$ is injective.
  But this follows directly because $S$ is firm.
\end{proof}

\begin{example}
  Let $S$ be a nondegenerate firm idempotent ideal of a nondegenerate firm commutative ring $R$.
  Then $\cat{FMod}_R\restrict{S}$ is monoidally equivalent to $\cat{FMod}_S$.
\end{example}
\begin{proof}
  Send $E$ in $\cat{FMod}_R\restrict{S}$ to $E$ with $S$-module structure $x \cdot s := xs$, and send an $R$-linear map $f$ to $f$. This defines a functor $\cat{FMod}_R\restrict{S} \to \cat{FMod}_S$.
  In the other direction, a firm $S$-module $F \simeq F \otimes_S S$ has firm $R$-module structure $(y \otimes s)\cdot r:=y \otimes (sr)$ because $S$ is idempotent, and if $g$ is an $S$-linear map then $g \otimes_S \id[S]$ is $R$-linear. This defines a functor $\cat{FMod}_S \to \cat{FMod}_R\restrict{S}$.
  Composing both functors sends a firm $R$-module $E$ to $E \otimes_S S \simeq E \otimes_R R \simeq E$, and a firm $S$-module $F$ to $F \otimes_S S \simeq F$.
\end{proof}

\section{Hilbert submodules}\label{sec:tensorsofhilbertmodules}

Fix a locally compact hausdorff space $X$.
The tensor product of Hilbert modules $E$ and $F$ is constructed as follows:
first, consider the algebraic tensor product of $C_0(X)$-modules,
and then complete it to a Hilbert module under the inner product $\inprod{x \otimes y}{x' \otimes y'}=\inprod{x}{x'}\inprod{y}{y'}$ .

Hilbert modules $E$ and $F$ are \emph{unitarily equivalent} if there are maps $u \colon E \to F$ and $u^\dag \colon F \to E$ satisfying $u\circ u^\dag = \id$, $u^\dag \circ u = \id$, and $\inprod{u(x)}{y}=\inprod{x}{u^\dag(y)}$ for all $x \in E$ and $y \in F$.

\begin{lemma} \label{lem:EtimesCo}
  Any Hilbert module $E$ is unitarily equivalent to $E \otimes C_0(X)$.
\end{lemma}
\begin{proof}\cite[p.~42]{lance1995hilbert}
  Define $\rho_E \colon E \otimes C_0(X) \to E$ by continuously linearly extending $x \otimes f \mapsto xf$. Then
  \[
    \|\sum_i x_i \otimes f_i \|^2 
    = \| \sum_{i,j} \inprod{x_i \otimes f_i}{x_j \otimes f_j} \| 
    = \| \sum_{i,j} \inprod{x_i}{x_j} f_i^* f_j \| 
    = \| \sum_{i,j} \inprod{x_i f_i}{x_j f_i} \| 
    = \| \rho_E(\sum_i x_i \otimes f_i) \|^2\text.
  \]
  Because the algebraic tensor product of $E$ and $C_0(X)$ is dense in $E \otimes C_0(X)$, this $C_0(X)$-linear map $\rho_E$ is isometric. Because $\{xf \mid x \in E, f \in C_0(X)\}$ is dense in $E$~\cite[page~5]{lance1995hilbert}, the map $\rho_E$ is surjective, and hence unitary~\cite[Theorem~3.5]{lance1995hilbert}.
\end{proof}

We now prove Theorem~\ref{thm:hilbertmodules}: there is a bijective correspondence between idempotent subunits in $\cat{Hilb}_{C_0(X)}$ and open subsets of $X$.

\begin{proof}[Proof of Theorem~\ref{thm:hilbertmodules}]
  If $U$ is an open subset of $X$, we may identify $C_0(U)$ with the closed ideal of $C_0(X)$ given by $\{f \in C_0(X) \mid f(X \setminus U)=0\}$: if $f \in C_0(U)$, then its extension by zero on $X \setminus U$ is in $C_0(X)$, and conversely, if $f \in C_0(X)$ is zero outside $U$, then its restriction to $U$ is in $C_0(U)$.  
  For idempotence, note that the canonical map $C_0(X) \otimes C_0(X) \to C_0(X)$ is always a unitary (see Lemma~\ref{lem:EtimesCo}), and hence the same holds for $C_0(U)$.
  Thus $C_0(U)$ is an idempotent subunit in $\cat{Hilb}_{C_0(X)}$.

  For the converse, let $s \colon S \rightarrowtail C_0(X)$ be an idempotent subunit in $\cat{Hilb}_{C_0(X)}$. We will show that $s(S)$ is a closed ideal in $C_0(X)$,  and therefore of the form $C_0(U)$ for some open subset $U \subseteq X$.
  It is an ideal because $s$ is $C_0(X)$-linear.
  To see that it is closed, let $g \in s(S)$. Then
  \begin{align*}
    \|g\|_S^4    
    & = \| \inprod{g}{g}_S^2 \|_{C_0(X)}
     = \| \inprod{g}{g}_S \inprod{g}{g}_S\|_{C_0(X)}
     = \| \inprod{g \otimes g}{g \otimes g}_{C_0(X)} \|_{C_0(X)}
     = \|g \otimes g\|_S^2 \\
    & \leq \|\rho_S^{-1}\|^2 \|g^2\|_S 
    = \|\rho_S^{-1}\|^2 \| \inprod{g}{g}_S g^* g \|_{C_0(X)}
    \leq \|\rho_S^{-1}\|^2 \|g\|_S^2 \|g\|_{C_0(X)}^2
  \end{align*}
  and therefore $\|g\|_S \leq \| \rho_S^{-1}\|^2 \|g\|_{C_0(X)}^2$.
  Because $s$ is bounded, it is thus an equivalence of normed spaces between $(S,\|-\|_S)$ and $(s(S), \|-\|_{C_0(X)})$. Since the former is complete, so is the latter.
\end{proof}

In contrast, not every subunit in $\cat{Hilb}_{C_0(X)}$ is induced by an open $U \subseteq X$.

\begin{example}
  Let $X=[0,1]$. If $f \in C_0(X)$, write $\hat{f} \in C_0(X)$ for the map $x \mapsto xf(x)$.
  Then $S=\{ \hat{f} \mid f \in E\}$ is a subobject of $E=C_0(X)$ in $\cat{Hilb}_{C_0(X)}$ under $\inprod{\hat{f}}{\hat{g}}_S = \inprod{f}{g}_E$ that is not closed under $\|-\|_E$.
\end{example}
\begin{proof}
  Clearly $S$ is a $C_0(X)$-module, and $\inprod{-}{-}_S$ is sesquilinear.
  Moreover $S$ is complete: $\hat{f_n}$ is a Cauchy sequence in $S$
  if and only if $f_n$ is a Cauchy sequence in $E$, in which case it converges in $E$ to some $f$, and so $\hat{f_n}$ converges to $\hat{f}$ in $S$.
  Thus $S$ is a well-defined Hilbert module.
  The inclusion $S \hookrightarrow E$ is bounded and injective, and hence a well-defined monomorphism.
  In fact, $E$ is a C*-algebra, and $S$ is an ideal. The closure of $S$ in $E$ is the closed ideal $\{f \in C_0(X) \mid f(0)=0\}$, corresponding to the closed subset $\{0\} \subseteq X$. It contains the function $x \mapsto \sqrt{x}$ while $S$ does not, and so $S$ is not closed. 
\end{proof}

\end{document}